\documentclass[12pt]{amsart}
\usepackage{amssymb,amsthm,amsmath,amsfonts}

\usepackage{geometry}
\usepackage{color}
\usepackage{amssymb}
\usepackage{verbatim}
\usepackage{amsmath}
\usepackage[T1]{fontenc}	
\usepackage[utf8]{inputenc}
\usepackage{enumerate}
\usepackage{amsfonts}
\usepackage{graphicx}
\usepackage{color}
\usepackage{bbm}

\newcommand{\mc}[1]{\mathcal{{#1}}}

\newcommand{\e}{\varepsilon}

\newcommand{\dd}{\mathrm{d}}

\newcommand{\1}{\textbf{1}}
\newcommand{\R}{\mathbb{R}}

\newcommand{\one}{\mathbbm{1}}

\DeclareMathOperator{\var}{Var}

\DeclareMathOperator{\vol}{vol}

\theoremstyle{definition}
\newtheorem{theorem}{Theorem} 
 
\newtheorem{proposition}[theorem]{Proposition}

\newtheorem{lemma}[theorem]{Lemma}

\theoremstyle{remark}
\newtheorem*{remark}{Remark}

\begin{document}

\newgeometry{tmargin=3cm, bmargin=3cm, lmargin=3cm, rmargin=3cm}

\title{Lower bounds on non-central sections of isotropic convex bodies }

\author{Jacek Jakimiuk}

\address{(JJ) University of Warsaw, Banacha 2, 02-097 Warsaw, Poland.}
\email{jj406165@mimuw.edu.pl}

\author{Daniel Murawski}

\address{(DM) University of Warsaw, Banacha 2, 02-097 Warsaw, Poland.}
\email{dk.murawski@student.uw.edu.pl}

\author{Piotr Nayar}
\thanks{Research of JJ and PN  was
		partially funded by the National Science Centre, Poland, grant 2024/55/B/ST1/02938}

\address{(PN) University of Warsaw, Banacha 2, 02-097 Warsaw, Poland.}
\email{nayar@mimuw.edu.pl}

\maketitle

\begin{abstract}
For fixed $t_0 \in [0,\sqrt{3}]$  we give asymptotically sharp lower bounds on the quantity $L_K \vol_{d-1}(K \cap H)$, where $H$ is a hyperplane at distance $t_0 L_K$ from the origin,  $K$ is any symmetric isotropic convex body in $\R^d$, and $L_K$ stands for the isotropic constant of $K$.   
\end{abstract}

\section{Introduction}

The investigation of sections of convex sets plays a central role in convex geometry. Much of the research in this direction has been driven by the so-called slicing problem of Bourgain from \cite{B86}, recently solved by Klartag and Lehec in \cite{KL25}. We refer the reader to the monograph \cite{BGVV14} for a comprehensive study of this topic and to a survey \cite{NT23} for sharp inequalities concerning sections of certain specific families of convex sets. 

Let us recall that $K \subseteq \R^d$ is called a convex body if $K$ is a compact convex set with non-empty interior. A convex body $K$ is (origin) symmetric if $K=-K$ and in this case it is called isotropic if it has volume $1$ and  covariance matrix proportional to the identity matrix, 
\[
\int_K x \otimes x \dd x = L_K^2 I_{d \times d},
\]
where the constant $L_K$ is called the isotropic constant of $K$. In \cite{MTW26} Melbourne, Tkocz and Wyczesany considered non-central sections of symmetric isotropic convex bodies. They proved that for every such body and every hyperplane $H$ in $\R^d$ with distance at most $L_K \sqrt{3}$ to the
origin, we have 
\[
\vol_{d-1}(K \cap H) \geq  L_K^{-1} e^{-\sqrt{6}} 2^{-\frac12}
\]  
and this lower bound is sharp. Here $\vol_{d-1}$ stands for the $(d-1)$-dimensional Lebesgue measure. In this article we extend their result by deriving sharp bounds for  hyperplanes at distance $t L_K$ from the origin for  fixed $t \in [0,\sqrt{3}]$.  By  $f_X$, $\sigma^2_X$ we will denote the density and variance of a random variable $X$, respectively. Our main result reads as follows.

\begin{theorem}\label{thm:slice}
Let $K \subseteq \R^d$ be an isotropic convex body.
\begin{enumerate}[(a)]
\item

 For every  $t_0 \in [0,\sqrt{3}]$ and every hyperplane $H$ at distance $t_0 L_K$ from the origin we have 
\[
\vol_{d-1}(K \cap H) \geq  L_K^{-1}   \inf_{b > 0} \sigma_{Y_b}f_{Y_b}(t_0\sigma_{Y_b}),
\]   
where for $b > 0$ by $Y_b$ we denote the random variable with density 
\[
f_{Y_b}(t) = \frac{e^{-|t|}}{2(1-e^{-b})} \one_{[-b, b]}(t).
\]
\item For $t_0 \in [\frac{1}{\sqrt2}, \sqrt{3}]$ we have
\[
	\vol_{d-1}(K \cap H) \geq  L_K^{-1} \frac{1}{\sqrt{2}}e^{-t_0\sqrt2}.
\]
  
\item For every $t_0 \in [0,c_0]$ where 
 \[
    	c_0 =  \inf_{b  >0}  \ \frac{1}{\sigma(b)} \log \left( \frac{\sqrt{3} \sigma(b)}{1-e^{-b}} \right), \textrm{where} \ \  \sigma(b)^2 =  2-\frac{e^{-b}b(b+2)}{1-e^{-b}}, \qquad c_0 \approx 0.63217
    \]
   we have 
   \[
   \vol_{d-1}(K \cap H) \geq  L_K^{-1} \frac{1}{2\sqrt{3}}.
   \]
\end{enumerate}
\end{theorem}

\begin{remark}
These results are asymptotically sharp as $d \to \infty$. In the case of point (b) the equality is achieved by the double cone, whereas  in point (c) the worst case is the cube, see Section \ref{sec:sharp} for the details. 
\end{remark}

For a unit vector $u$ let us introduce the section function 
\[
f(t)= \vol_{d-1}(K \cap (u^\perp + t u)). 
\]  
For an isotropic convex body $K$ we clearly  have $\int f(t) \dd t =1$ and $\int t^2 f(t) \dd t = L_K^2$. Therefore $f$ is a density of some random variable $X_0$ with variance $L_K^2$. Moreover, it is a classical fact following from Brunn-Minkowski inequality that $X_0$ is log-concave, that is $f=e^{-V}$ for some convex function $V:\R \to \R \cup \{\infty\}$.  We can now see that proving Theorem \ref{thm:slice} boils down to estimating $ \sigma_Xf(t_0\sigma_X)$ for a log-concave symmetric density $f$. 

 Let $\mathcal S$ be the set of all nonzero real symmetric log-concave random variables. We assume that all densities are upper-semicontinuous. In fact every random variable in $S$ admits a unique upper-semicontinuous density, with possible discontinuity points only on the two boundary points of their support. By $\mc{S}_M$ we shall denote the set of densities from $\mc{S}$ having variance one and support in the interval $[-M,M]$. The equivalent one-dimensional reformulation of Theorem \ref{thm:slice} reads as follows.   

\begin{theorem}\label{thm:main}
For $b > 0$ let $Y_b$ be a random variable with density 
\[
f_{Y_b}(t) = \frac{e^{-|t|}}{2(1-e^{-b})} \one_{[-b, b]}(t).
\]
\begin{enumerate}[(a)]
    \item For $t_0 \in [0, \sqrt{3}]$ we have
    $$\inf_{X \in S} \sigma_Xf(t_0\sigma_X) = \inf_{b > 0} \sigma_{Y_b}f_{Y_b}(t_0\sigma_{Y_b}).$$
    \item We have
    $$\inf_{X \in S} \sigma_Xf(t_0\sigma_X) = \lim_{b \to \infty} \sigma_{Y_b}f_{Y_b}(t_0\sigma_{Y_b}) = \frac{1}{\sqrt{2}}e^{-t_0\sqrt2}$$
    if and only if $t_0 \in [\frac{1}{\sqrt2}, \sqrt{3}]$. This infimum is achieved for the Laplace distribution.
    \item There is a constant $c_0$, such that we have
    $$\inf_{X \in S} \sigma_Xf(t_0\sigma_X) = \lim_{b \to 0} \sigma_{Y_b}f_{Y_b}(t_0\sigma_{Y_b}) = \frac{1}{2\sqrt{3}}$$
    if and only if $t_0 \in [0, c_0]$, where $\frac12 < c_0 < \frac1{\sqrt2}$. This infimum is achieved for the uniform distribution. The constant $c_0$ is given by
    \[
    	c_0 =  \inf_{b  >0}  \ \frac{1}{\sigma(b)} \log \left( \frac{\sqrt{3} \sigma(b)}{1-e^{-b}} \right), \textrm{where} \ \  \sigma(b)^2 =  2-\frac{e^{-b}b(b+2)}{1-e^{-b}}
    \]
    and its numerical value is $c_0 \approx 0.63217$.
\end{enumerate}
\end{theorem}

\noindent We note that the quantity $ \sigma_X f(t_0\sigma_X)$ is scale-invariant, that is, does not change under scaling $X \to \lambda X$ for any $\lambda>0$.

The article is organized as follows. In Section \ref{sec:sharp} we construct convex bodies that asymptotically  give equality in Theorem \ref{thm:slice}.  In Section \ref{sec:reduction} we recall the reduction scheme allowing us to assume that our densities are log two-piece affine. The last three sections are devoted to the proof of points (a), (b) and (c) of Theorem \ref{thm:main}, respectively. 

\section{Equality cases}\label{sec:sharp}

Let $X$ be a symmetric log-concave random variable with density $f=e^{-V}$. We shall show that the quantity $\sigma_X f(t_0 \sigma_X)$ can be obtained as a limit for $n \to \infty$ of expressions of the form 
\[
f_n(t_0): = L_{K_n} \vol_{n}(K_n \cap \{(x,y) \in \R \times \R^n: x= t_0 L_{K_n}\}),
\] 
where $K_n \subseteq \R^{n+1}$ is some sequence of isotropic convex bodies. We take
\[
	K_n = \left\{ (x,y) \in \R \times \R^n: |y| \leq a\left( 1 - \frac{V(bx)}{n} \right)_+     \right\}.
\] 
It is a straightforward computation to check that $K_n$ is convex. Following a similar reasoning to that in \cite{MTW26}, our goal is now to choose positive parameters $a,b$ such that $K_n$ becomes isotropic. Let $v_n$ denote the volume of the unit ball in $\R^n$ and let $c_{k,l}=\int_{\R} s^k (1-V(s)/n)_+^{n+l}$. We have
\[
	\vol_{n+1}(K_n) = v_n a^n \int_{\R} \left( 1 - \frac{V(bx)}{n} \right)_+^n \dd x =  v_n a^n b^{-1} c_{0,0}.
\]
Moreover 
\[
	\int_{K_n} x^2 \dd x \dd y = v_n a^n \int_{\R} x^2 \left( 1 - \frac{V(bx)}{n} \right)^n \dd x = v_n a^n b^{-3} c_{2,0}.
\]
Finally, due to the symmetry of $K_n$ we have that for every $i=1,\ldots,n$
\begin{align*}
	\int_{K_n} y_i^2 \dd y \dd x & = \frac{1}{n} \int_{K_n} |y|^2 \dd y \dd x = v_n \int_0^{\infty} r^{n+1} \1_{\left\{ r \leq a\left( 1 - \frac{V(bx)}{n} \right)_+ \right\}}  \dd r \dd x \\
	& = \frac{v_n a^{n+2}}{n+2} \int_\R\left( 1 - \frac{V(bx)}{n} \right)_+^{n+2} \dd x =  \frac{v_n a^{n+2}b^{-1}}{n+2} c_{0,2}.
\end{align*}
Since the covariance matrix of $K_n$ is diagonal due to the symmetries, we only need to ensure that
\[
	 v_n a^n b^{-1} c_{0,0} = 1, \qquad   b^{-2} c_{2,0} =  \frac{ a^{2}}{n+2} c_{0,2}.
\]
and in this case we have
\[
	L_{K_n}^2 = v_n a^n b^{-3} c_{2,0}.
\]
Solving the above system of equations for $a,b$ we arrive at 
\[
	a = (c_{0,0} v_n)^{- \frac{1}{n+1}} \left( \frac{c_{2,0}}{c_{0,2}}(n+2) \right)^{\frac{1}{2(n+1)}}, \qquad  b = (c_{0,0} v_n)^{ \frac{1}{n+1}} \left( \frac{c_{2,0}}{c_{0,2}}(n+2) \right)^{\frac12 -\frac{1}{2(n+1)}}.
\]
We have the following asymptotic relations as $n \to \infty$,
\[
	c_{2,0} \sim \sigma_X^2, \quad c_{0,0}, c_{0,2} \sim 1, \quad v_n^{- \frac{1}{n+1}} \sim \sqrt{\frac{n}{2\pi e}}, \quad a \sim \sqrt{\frac{n}{2\pi e}}, \quad b \sim \sigma_X \sqrt{2\pi e}.  
\]
We also have
\[
	v_n a^n = v_n (c_{0,0} v_n)^{- \frac{n}{n+1}}   \left( \frac{c_{2,0}}{c_{0,2}}(n+2) \right)^{\frac{n}{2(n+1)}} \sim v_n^{\frac{1}{n+1}} \sigma_X \sqrt{n}  \sim  \sigma_X  \sqrt{2\pi e}
\]
and 
\[
	L_{K_n}^2 = v_n a^n b^{-3} c_{2,0} \sim   \sigma_X  \sqrt{2\pi e} \cdot (\sigma_X \sqrt{2\pi e})^{-3} \sigma_X^2  = \frac{1}{2\pi e}, \quad L_{K_n} b \sim \sigma_X, \quad L_{K_n} v_n a^n \sim \sigma_X.
\]
Therefore 
\[
	f_n(t_0) = L_{K_n} v_n a^n \left( 1 - \frac{V(t_0  L_{K_n} b)}{n}\right)_+^n \xrightarrow[n \to \infty]{} \sigma_X e^{-V(t_0 \sigma_X)} = \sigma_X f(t_0 \sigma_X).
\]
The function $V(x)=2|x|$ (exponential random variable) corresponds to $K_n$ being a double cone and $V(x) = \log 2$ for $|x| \leq 1$ and $V(x)=\infty$ for $|x|>1$ (uniform distribution) to the cylinder, but clearly the same section function can be obtained just by taking the cube.

\section{Reduction scheme}\label{sec:reduction}

We begin with the following well known reduction scheme based on the localization technique developed by Fradelizi and Gu\'edon in \cite{FG06}.

\begin{lemma}\label{lem:degrees_of_freedom}
    In the set $\mc{S}_M$ the infimum of the quantity $f(t_0)$ is attained for some random variable with density of the form
    \[
    f_{a,b,c,\gamma}(x)  = c\left(\one_{[0, a)}(|x|) + e^{-\gamma (|x|-a)} \one_{[a, a+b]}(|x|)\right) = c\min\left(1, e^{-\gamma(|x|-a)}\right)\one_{[0,a+b]}(|x|)
    \]
    for some parameters $a, b, c, \gamma \geq 0$, where $a>0$ or $b>0$.
\end{lemma}
\begin{proof}[Proof without heavy lifting.]
We follow the ideas from \cite{MNT21}. For $\e>0$ let us consider the functional $\Phi(f)=f(t_0)+\e \int \sqrt{f}$. This functional is strictly concave. Moreover, it follows from Lemma 10 in \cite{MNT21} that $f(0)$ is uniformly bounded on $\mc{S}_M$ and thus $\Phi(f)$ is also bounded. Let $(f_n) \subset \mc{S}_M$ be a sequence  satisfying $\Phi(f_n) \to \inf_{f \in \mc{S}_M} \Phi(f)$. According to Lemma 12 in \cite{MNT21} without loss of generality we can assume that $(f_n)$ is pointwise convergent to a certain function $f \in \mc{S}_M$. Thus by the Lebesgue dominated convergence theorem the above infimum is attained on $f$. Since $\Phi$ is strictly convex, the function $f$ must be an extremal point of $\mc{S}_M$.  Indeed, if $f=\lambda f_1 + (1-\lambda) f_2$ for some $\lambda \in (0,1)$ and distinct $f_1, f_2 \in \mc{S}_M$ then 
\[
\Phi(f) > \lambda \Phi(f_1) + (1-\lambda) \Phi(f_2) \geq \min(\Phi(f_1), \Phi(f_2)).
\]
Finally, Steps III and IV of the proof of Theorem 1 in \cite{MNT21} show that $f$ has to be of the desired form.

\emph{Proof with heavy lifting.} 
  We argue similarly to \cite{FG06}. By Prokhorov's theorem the set of all probability measures on $I=[-M,M]$ is weak-star compact. By Theorem 2.2 of \cite{B74} the set of log-concave measures on $I$ is closed in weak-star topology, so is the set of symmetric measures. The condition $\int x^2 d\mu = 1$ is also closed. Therefore, the set $\mc{S}_M$ is compact in weak-star topology. The mapping $\varphi$ defined by $\mu \mapsto \frac{d\mu}{dx}(t_0)$ is linear on $\mc{S}_M$, in particular it is concave. Therefore, by the classical Bauer Maximum Principle $\inf \{\varphi(\mu):  \mu \in \mathrm{conv}(\mc{S}_M)\}$ is attained at some extreme point of $\mc{S}_M$, which has to be of the desired form, as in the first argument.
\end{proof}

\noindent Since all even densities on $\R$ can be approximated by compactly supported densities, and by homogeneity we can always assume $\sigma_X=1$, in order to prove Theorem \ref{thm:main} one can focus on densities of the form from Lemma \ref{lem:degrees_of_freedom}. Again by homogeneity of the quantity $\sigma_Xf_X(t_o\sigma_X)$ we can now work under the assumption $\gamma=1$ instead of $\sigma_X = 1$. We can also calculate that in this case we have $c = c(a,b) = \left(2(a+1-e^{-b})\right)^{-1}$. Thus, the density is of the form
\[
f_{a,b}(x) =  \left(2(a+1-e^{-b})\right)^{-1}\min\left(1, e^{-(|x|-a)}\right)\one_{[0,a+b]}(|x|).
\]

\section{Proof of Theorem \ref{thm:main}(\emph{a})}

Be begin with the following well-known lemma.

\begin{lemma}\label{lem:intersections}
    Let $X, Y$ be real symmetric random variables with density functions $f, g$. If there exists $t > 0$ such that $f(x) \leq g(x)$ on $(0,t)$ and $f(x) \geq g(x)$ on $(t, \infty)$, then $\var X \geq \var Y$. The inequality is strict unless $f(x) = g(x)$ almost everywhere.
\end{lemma}
\begin{proof}
We have
\begin{align*}
    \frac12(\var X - \var Y) &= \int_0^t (f(x)-g(x))x^2dx + \int_t^\infty (f(x) - g(x))x^2dx \\
    &\geq \int_0^t (f(x)-g(x))t^2dx + \int_t^{\infty}(f(x)-g(y))t^2dx \\
    &= t^2\int_0^{\infty}f(x)-g(x)dx = 0.
\end{align*}
We note that the inequality is strict unless $f(x)=g(x)$ almost everywhere.
\end{proof}

Let 
\[
\sigma^2 = \sigma^2(a,b) = 2c(a,b)\left(\int_0^a x^2dx + \int_a^{a+b}e^{-(x-a)}x^2dx\right)
\] 
be the variance of $f_{a,b}$. 

\begin{lemma}\label{lem:variance}
    The variance $\sigma^2(a,b)$ is non-decreasing both with respect to $a$ and with respect to $b$. Moreover, we have $\sqrt{3} \sigma(a,b) \leq a+b$.
\end{lemma}
\begin{proof}
    Note that $c(a,b)$ is decreasing both with respect to $a$ and $b$. Let us now fix $b$ and we will prove that $\sigma^2$ is increasing in $a$. Take some $a'>a$. Consider functions $V_1 = \log f_{a,b}$, $V_2 = \log f_{a',b}$ (we put $\log(0) = -\infty$. Both functions are first constant and then affine with slope $-1$ on an interval of length $b$. Thus, $V_1  - V_2$ is non-increasing on $[0, a+b]$ and $V_2 \geq V_1$ on $[a+b, \infty)$. Therefore either there exists $t_0 \in [0, a+b)$ such that $V_1 \geq V_2$ on $[0, t_0]$ and $V_1 \leq V_2$ on $[t_0, \infty)$ or $t_0 = a+b$ satisfies this condition. By Lemma \ref{lem:intersections} this shows that $\sigma^2(a',b) > \sigma^2(a,b)$.
    
    We will now prove that $\sigma^2(a,b)$ is increasing in $b$. Note that if $b'>b$, then $c(a, b') < c(a,b)$ and the ratio $f_{a,b}/f_{a'b}$ is constant and greater than $1$ on $[0, a+b]$. On the other hand, we have $f_{a,b'} \geq f_{a,b}=0$ on $(a+b, \infty)$, so the desired inequality is again true by Lemma \ref{lem:intersections}.
    
    To prove the second part it suffices to apply Lemma \ref{lem:intersections} to $Y$ with density $f_{a,b}$ and $X$ being uniform on $[a,b]$ as $\var(X)=\frac{1}{\sqrt{3}}(a+b)$.
\end{proof}

We now fix $t_0$ in Theorem \ref{thm:main} and we notice that according to the second part of Lemma \ref{lem:variance} the quantity $\sigma f_{a,b}(t_0 \sigma)$ is equal to
\[
g(a,b) = g_{t_0}(a,b) =  \frac{\sigma}{2(a+1-e^{-b})}\min\left(1, e^{-(t_0\sigma-a)}\right),
\]
where 
\[
\sigma^2 = \sigma^2(a,b) = 2c(a,b) \left(\int_0^a x^2dx + \int_a^{a+b}e^{-(x-a)}x^2dx\right)
\] 
is the variance of the random variable given by the density function $f_{a,b}$. This is well defined as long as $a > 0$ or $b > 0$. It is therefore enough to establish the following proposition.
\begin{proposition}\label{prob:thm1a}
For any nonnegative $a,b$ such that $a+b>0$ we have
\begin{equation}\label{eq:goal}
    g(a,b) \geq \inf_{b > 0} g(0,b).
\end{equation}
\end{proposition}

Our first step is the following lemma.

\begin{lemma} \label{lem:constant_part}
    If $t_0\sigma(a,b) \leq a$, then inequality (\ref{eq:goal}) holds. 
\end{lemma}
\begin{proof}
    Let $a' = \sqrt{3}\sigma(a,b)$, so that $\sigma(a',0) = \sigma(a,b) =: \sigma$. Since $f_{a,b}-f_{a',0}$ integrates to $0$, it has to change sign at least once. But since $f_{a,b}$ and $f_{a',0}$ have the same variance, by Lemma \ref{lem:intersections} the difference must change sign in at least two points in $[0,\infty)$. Note that $f_{a',0}$ is uniform and $f_{a,b}$ is non-increasing on $[0,\infty)$ and thus two sign-change points are possible only if $f_{a',0}(0) \leq f_{a,b}(0)$ which gives $c(a',0) \leq c(a,b)$. Thus, 
\[
g(a',0) = \sigma f_{a',0}(t_0\sigma) = \sigma c(a',0) \leq  \sigma c(a,b) = \sigma f_{a,b}(t_0\sigma) = g(a,b).
\]
 Now we just note that 
\[
g(a',0) = \frac{\sigma}{2\sqrt{3}\sigma}=\frac{1}{2\sqrt{3}} = \lim_{b \to 0} g(0,b)  \geq \inf_{b>0} g(0,b),
\]       
where the last equality is a direct computation.
\end{proof}

\noindent The above lemma shows that inequality (\ref{eq:goal}) holds whenever $t_0\sigma$ lies in the interval where $f_{a,b}$ is constant. 
Thus, it remains to show that
\[
h(a,b) := h_{t_0}(a, b) = \frac{\sigma}{2(a+1-e^{-b})}e^{a-{t_0\sigma}} \geq \inf_{b>0} g(0,b)
\]
whenever $t_0\sigma > a$. Note that $\sigma$ and as a result also $h$ have limits when $b$ goes to infinity, thus it makes sense to extend the domain of $h$ to $B = [0, \infty) \times [0, \infty] \setminus \{(0, 0)\}$.

\begin{lemma}\label{lem:near_zero}
    The function $h$ can be extended to $[0, \infty) \times [0, \infty]$ by putting $h(0, 0) = \frac{1}{2\sqrt{3}}$, so that the resulting function is continuous at $(0, 0)$. 
\end{lemma}
\begin{proof}
    Let $(a_n, b_n) \to (0, 0)$ as $n \to \infty$ and $(a_n, b_n) \in B$ for all $n$. Clearly $\sigma(a_n,b_n) \leq a_n+b_n \to 0$.  It is enough to show that $g(a_n,b_n) \to \frac{1}{2\sqrt{3}}$ since 
    \[
    \frac{g(a_n,b_n)}{h(a_n,b_n)} \to \min(e^{t_0 \sigma(a_n,b_n)},1) \to 1.
    \]
   
    Let $X_n$ be the random variable with density $f_{a_n,b_n}$ and consider $Y_n = \frac{1}{a_n +b_n}X_n$ which is a random variable supported on $[-1, 1]$. By scale invariance  we have 
    \[
     g(a_n,b_n) = \sigma_n f_n(t_0 \sigma_n),
    \]
where $f_n =  f_{\alpha_n, \beta_n, c_n, \gamma_n}$ so that parameters $\alpha_n, \beta_n, c_n, \gamma_n$ describe density of $Y_n$ and $\sigma_n^2 = \var Y_n$.  Note that for all $x$ in support of $f_{a_n, b_n}$ we have 
\[
f_{a_n, b_n}(0) \geq f_{a_n, b_n}(x) \geq e^{-(a_n+b_n)}f_{a_n, b_n}(0)
\]
and since this inequality is scale invariant, it also holds for $f_n$ and so the ratio of $\max_{[-1, 1]} f_n$ and $\min_{[-1, 1]}f_n$ is bounded by $e^{-(a_n + b_n)}$ which goes to $1$ as $n \to \infty$. Since $\int_{-1}^1 f_n(x) dx = 1$ it easily follows that $f_n$ converge uniformly to $\frac12 \one_{[-1,1]}$. Thus, $\sigma_n \to \frac{1}{\sqrt{3}}$ as $n \to \infty$ and $\sigma_nf(t_0\sigma_n) \to \frac{1}{2\sqrt{3}}$.
\end{proof}

\noindent From now on, we will consider this extended function $h$.

\begin{lemma}\label{lem:reduction}
    Let us denote $A = \{(a, \infty): a \geq 0\} \cup \{(a, 0): a \geq 0\} \cup \{(0, b): b > 0\}$. Then $\inf\{h(a,b): (a, b) \in B\} = \inf \{h(a, b): (a, b) \in A\}$.
\end{lemma}
\begin{proof} 
    We denote $c(t_0) = \inf_{A} h(a,b)$. We will also use notation $F_a = \frac{d}{da} F$ and $F_b = \frac{d}{db} F$ for partial derivatives of functions with respect to $a,b$. We differentiate $h(a,b)$.
\begin{align*}
    \frac{d}{da}h(a,b) &=  \frac{h(a,b)}{\sigma(a+1-e^{-b})}\left(\left(\sigma_a + \sigma (1-t_0\sigma_a)\right)(a+1-e^{-b}) - \sigma\right) \\
    &= \frac{h(a,b)}{\sigma^2(a+1-e^{-b})}\left(\frac12(\sigma^2)_a (1-t_0\sigma)(a+1-e^{-b}) + (a-e^{-b})\sigma^2\right),\\
    \frac{d}{db} h(a,b) &= \frac{h(a,b)}{\sigma(a+1-e^{-b})}\left((\sigma_b - t_0\sigma\sigma_b)(a+1-e^{-b}) - e^{-b}\sigma\right) \\
    &= \frac{h(a,b)}{\sigma^2(a+1-e^{-b})}\left(\frac12 (\sigma^2)_b (1-t_0\sigma)(a+1-e^{-b}) - e^{-b}\sigma^2\right).
\end{align*}
From this and Lemma \ref{lem:variance} we can see that if $t_0\sigma \geq 1$, then $\frac{d}{db} h(a,b) < 0$. Since $\sigma$ is increasing in $b$, this inequality holds in all points $h(a, b')$ with $b' \geq b$ and thus $h(a, b) > h(a, \infty) \geq c(t_0)$. By Lemma \ref{lem:constant_part} we can assume that $t_0\sigma \geq a$, otherwise 
\[
h(a, b) \geq g(a,b) \geq \inf_{b>0} g(0,b) = \inf_{b>0} h(0,b) \geq c(t_0).
\]
If additionally $a \geq 1$ then we have $t_0\sigma \geq 1$ and thus $h(a,b) \geq c(t_0)$. From now on, we shall assume that $a <1$ and $t_0\sigma < 1$.

Suppose that $(a, b) \in \mathrm{int}B$ is a critical point of $h$. Then, we must have $\frac{d}{da}h(a,b) = \frac{d}{db} h(a, b) = 0.$ Note that the second equality implies $1-t_0\sigma \neq 0$. We have
\begin{equation}\label{eq:crit_point}
    -(\sigma^2)_a  = \frac{2(a-e^{-b})\sigma^2}{(1-t_0\sigma)(a+1-e^{-b})} =  e^b(a-e^{-b})(\sigma^2)_b.
\end{equation}
We will now calculate the appropriate derivatives of $\sigma^2$. We have
$$\sigma^2 = 2 + \frac{a^3 + 3a^2 - 3e^{-b}(a+b)(a+b+2)}{3(a+1-e^{-b})}.$$
Thus,
\begin{align*}
(\sigma^2)_a &= (3(a+1-e^{-b})^{2})^{-1}\Big((3a^2 + 6a - 6e^{-b}(a + b + 1))(a+1-e^{-b}) \\
&-(a^3 + 3a^2 -3e^{-b}(a+b)(a+b+2))\Big),\\
    (\sigma^2)_b &= e^{-b}(3(a+1-e^{-b})^{2})^{-1}\Big(3((a+b)^2 - 2)(a+1-e^{-b}) \\
    &- (a^3 + 3a^2 -3e^{-b}(a+b)(a+b+2))\Big).
\end{align*}
With purely algebraical manipulations one can see that
\[
(\sigma^2)_a + e^b( a-e^{-b})(\sigma^2)_b = \frac{a(2a^2 + 6ab + 3b^2)}{3(a+1-e^{-b})}.
\]
This is positive for $a>0$, thus equality (\ref{eq:crit_point}) cannot hold and therefore there are no critical points of $h(a,b)$ in $\mathrm{int} \, B$. 

Let us now assume that the minimum of $h$ is attained on $[0,1] \times [0,\infty]$ and take a sequence of points $(a_n, b_n) \in B$ such that $h(a_n, b_n)$ is decreasing and converges to $c(t_0)$ and assume that $a_n \leq 1$ for all $n$. By passing to a subsequence, we might assume that $(a_n,b_n) \to (a,b) \in [0,1] \times [0,\infty]$. Since $h$ does not have critical points in $\mathrm{int} \, B$, we get that either $(a,b) \in A$, in which case the proof is finished, or that $a=1$, which leads to $h(1, b) \geq c(t_0)$ by the previous part.   
 \end{proof}

 \begin{proof}[Proof of Proposition \ref{prob:thm1a}]
   By Lemma \ref{lem:constant_part} and Lemma \ref{lem:reduction} we only need to check the inequality $h(a,b) \geq \inf_{b>0} g(0,b)$  in the set $A$. Since $h(0,b)=g(0,b)$ the inequality clearly holds for $a = 0$.
      If $b = 0$, then the distribution is uniform and we have $h(a,0) \geq g(a,0) =   \frac{1}{2\sqrt{3}} = \lim_{b \to 0} g(0,b)$.
      
     We now turn to the case $b = \infty$. We have
    \[
    \sigma^2(a, \infty) =  \frac{a^3+3a^2+6a+6}{3(a+1)}, \quad (\sigma^2)_a(a, \infty) = \frac{2a(a^2+3a+3)}{3(a+1)^2}.
    \]
     We will prove the inequality $\frac{d}{da}h(a, \infty) \geq 0$. It reads
     $$\frac{d}{da}h(a, \infty) = \frac{h(a, \infty)}{\sigma^2(a+1)}\left(\frac12(\sigma^2)_a(1-t_0\sigma)(a+1) + a\sigma^2\right) \geq 0.$$
     We notice that it is enough to prove the inequality at the largest possible value of $t_0$, that is $\sqrt{3}$. After canceling positive common factors we get an equivalent form
    \[
    (\sigma^2)_a(1-\sqrt{3}\sigma)(a+1) + 2 a\sigma^2 \geq 0.
    \]
    This is
    \[
    	\frac{2a(a^2+3a+3)}{3(a+1)}\left (1- \sqrt{ \frac{a^3+3a^2+6a+6}{a+1}} \right) +  \frac{2a(a^3+3a^2+6a+6)}{3(a+1)} \geq 0.
    \]
    Canceling the common factor $\frac{2a}{3(a+1)}$ and rearranging yields
    \[
    (a^3 + 4a^2 + 9a + 9)\sqrt{a+1} \geq (a^2+3a+3)\sqrt{ a^3+3a^2+6a+6},
    \]	
which is true since
\[
	 (a^3 + 4a^2 + 9a + 9)^2(a+1) - (a^2+3a+3)^2( a^3+3a^2+6a+6) = 3 a^5+19 a^4+54 a^3+90 a^2+81 a+27 > 0.
\]   
Thus, for all $t_0 \in [0, \sqrt{3}]$ and all $a > 0$ we have $h_{t_0}(a, \infty) \geq h_{t_0}(0, \infty) = g(0,\infty)$.
 \end{proof}

\section{Proof of Theorem \ref{thm:main}(\emph{b})}

\noindent Now that we have shown the first part of Theorem \ref{thm:main}, we can restrict ourselves to random variables $Y_b$. We will use notation $\sigma(b) = \sigma(0,b) = \sqrt{2-\frac{e^{-b}b(b+2)}{1-e^{-b}}}$. For $b > 0$ let
\begin{align*}
    G_{t_0}(b) &= g_{t_0}(0, b) =  \sigma_{Y_b}f_{Y_b}(t_0\sigma_{Y_b}) = \sigma(b)\frac{e^{-t_0\sigma(b)}}{2(1-e^{-b})} \\
    &= \frac{1}{2(1-e^{-b})}\sqrt{2 - \frac{e^{-b}b(b+2)}{1-e^{-b}}}\exp\left(-t_0 \sqrt{2 - \frac{e^{-b}b(b+2)}{1-e^{-b}}}\right).
\end{align*}
We also put $G_{t_0}(0) = \frac{1}{2\sqrt{3}}$ and $G_{t_0}(\infty) = \frac1{\sqrt{2}}e^{-\sqrt{2}t_0}$. Then $G_{t_0}$ is continuous on $[0, \infty]$.

    It is enough to show that $G_{t_0}(b)' \leq 0$ for all $b > 0$ if $t_0 \geq \frac{1}{\sqrt{2}}$ and that $G_{t_0}(b)' > 0$ for sufficiently large $b$ (dependent on $t_0$) for $t_0 < \frac{1}{\sqrt{2}}$.

    \noindent We have 
    \[
    G_{t_0}(b)' = \frac{G_{t_0}(b)}{(e^b-1)\sigma(b)^2}\left(\frac12\left(\sigma(b)^2\right)'(1-t_0\sigma(b))(e^b-1) - \sigma(b)^2\right).
    \]
    Thus, $G_{t_0}(b)' \leq 0$ if and only if
    \begin{equation}\label{eq:pałowanie1}
        \left(\sigma(b)^2\right)'(e^b-1)(1-t_0\sigma(b)) - 2\sigma(b)^2 \leq 0.
    \end{equation}
    We have
    \[
    \left(\sigma(b)^2\right)' = \frac{e^b(b^2 - 2)+ 2b +2}{(e^b-1)^2}.
    \]
    The inequality (\ref{eq:pałowanie1}) now takes the form
    \begin{align*}
        &(e^b(b^2 - 2)+ 2b +2)(1-t_0\sigma(b)) - (e^b-1)\left(2 - \frac{b(b+2)}{e^b-1}\right) \\
        =& (e^b(b^2-2) + 2b+2)(1-t_0\sigma(b)) - 2\left(2e^b - 2 - b(b+2)\right) \leq 0.
    \end{align*}
    We can see that $\lim_{b \to \infty} \sigma(b) = \sqrt{2}$ and $\sigma(b)$ is increasing. Thus, for $t_0 < \frac{1}{\sqrt{2}}$ we have
    \begin{align*}
        (e^b(b^2-2) + 2b+2)(1-t_0\sigma(b)) - 2\left(2e^b - 2 - b(b+2)\right) \geq e^b(b^2-2)(1 - t_0\sqrt{2}) - 4 e^b.
    \end{align*}
    If $t_0 < \frac{1}{\sqrt{2}}$, then this is positive for large enough $b$, thus the function $G_{t_0}$ is eventually increasing. This means that there exists a finite $b_0 \geq0$ such that $G_{t_0}(b_0) < G_{t_0}(\infty)$.

    We will now show that for $t_0 \geq \frac{1}{\sqrt{2}}$ we have $G_{t_0}'(b) \leq 0$ for all $b > 0$.
    We can see that the expression on the left side of (\ref{eq:pałowanie1}) is decreasing in $t_0$. Thus, it is enough to show non-positivity of $G_{1/\sqrt{2}}'$. That is, we shall prove that
    \begin{equation}\label{eq:pałowanie_sqrt2}
        \left(\sigma(b)^2\right)'(e^b-1)\left(1-\frac{\sigma(b)}{\sqrt{2}}\right) - 2\sigma(b)^2 \leq 0.
    \end{equation}
    Since $[\sigma(b)^2]', e^b-1, \sigma(b) \geq 0$, after rearranging terms and squaring this will follow from
    \[
    2\left([\sigma(b)^2]'(e^b-1) - 2\sigma(b)^2\right)^2 \leq
    \left([\sigma(b)^2]'\right)^2(e^b-1)^2\sigma(b)^2.
    \]
    Using formulas for $\sigma(b)^2, [\sigma(b)^2]'$ and canceling the $(e^b-1)$ terms our target inequality takes the form
    \[
    2(e^b-1)\left(e^b(b^2-6) + 2b^2 + 6b + 6\right)^2 \leq \left(e^b(b^2-2) + 2b + 2\right)^2 \cdot \left(2e^b - (b^2 + 2b +2)\right).
    \]
    We rewrite the above inequality as
    \[
    P_3(b)e^{3b} + P_2(b)e^{2b} + P_1(b)e^b + P_0(b) \geq 0,
    \]
    where $P_0, P_1, P_2, P_3$ are polynomials given by
    \begin{align*}
        P_0(b) &= 4b^4 + 32b^3 + 92b^2 + 120b + 64, \\
        P_1(b) &= -4b^5 - 12b^4 - 32b^3 - 120b^2 - 240b - 192,\\
        P_2(b) &= -b^6 - 2b^5 -4b^4 -8b^3 + 12b^2 +120b + 192, \\
        P_3(b) &= 16b^2 - 64.
    \end{align*}
    All of the above polynomials are divisible by $b+2$. We define $Q_i(b) = P_i(b)/(b+2)$ and we have
    \begin{align*}
        Q_0(b) &= 4b^3 + 24b^2 + 44b + 32,\\
        Q_1(b) &= -4b^4-4b^3-24b^2-72b-96,\\
        Q_2(b) &= -b^5-4b^3+12b+96,\\
        Q_3(b) &= 16b-32.
    \end{align*}
    Now we write $Q_3(b)e^{3b} + Q_2(b)e^{2b} + Q_1(b)e^b + Q_0(b) = \sum_{n=0}^{\infty} q_nb^n$. Since $b\geq 0$, the theorem will follow from $q_n \geq 0$ for all $n$. By direct calculation we have $q_0=q_1=q_2=q_3=q_4 = 0$. For $n \geq 5$ we have the formula 
    \begin{align*}
       n!q_n = & -4 n^4+20 n^3-56 n^2-32 n-96  \\
       & + 2^n\left(-\frac{n^5}{32}+\frac{5 n^4}{16}-\frac{51 n^3}{32}+\frac{49 n^2}{16}+\frac{17 n}{4}+96\right) \\
       & + 3^n\left(\frac{16n}{3} - 32\right).
    \end{align*}
    We verify using direct numerical computation that for $n\leq 25$ these are non-negative (in fact, $q_n = 0$ for $n \leq 8$). For $n > 25$ we have
    \[
    -4 n^4+20 n^3-56 n^2-32 n-96  = -4 n^4+ (18 n^3-56 n^2) + (n^3-32 n) + (n^3-96)  \geq  -4 n^4.
    \]
    Moreover
    \[
    	-\frac{n^5}{32}+\frac{5 n^4}{16}-\frac{51 n^3}{32}+\frac{49 n^2}{16}+\frac{17 n}{4}+96 >  -\frac{n^5}{32}+\frac{5 n^4}{16}-\frac{51 n^3}{32} = -\frac{n^5}{32}+ \frac{n^3}{32} \left(10n - 50\right) >  -\frac{n^5}{32}
    \]
and $ \frac{16n}{3} - 32 > 2n$. This gives
\[
	 n!q_n  > -4n^4 - \frac{n^5}{32} \cdot 2^n + 3^n \cdot 2n > - \frac{n^5}{16} \cdot 2^n + 3^n \cdot 2n > 0.  
\]

\section{Proof of Theorem \ref{thm:main}(\emph{c})}

The largest $t_0$ for which the inequality 
\[
G_{t_0}(b) = \frac{1}{2(1-e^{-b})} \sigma(b) \exp(-t_0 \sigma(b) ) \geq \frac{1}{2\sqrt{3}}
\]
holds for all $b > 0$ is given by
\[
	c_0 = \inf_{b  >0}  \ \frac{1}{\sigma(b)} \log \left( \frac{\sqrt{3} \sigma(b)}{1-e^{-b}} \right) \approx 0.63217.
\]
Now we will give a formal proof that the uniform distribution is optimal for $t_0 \leq \frac12$.

We proceed similarly as in the previous part of the proof. We want to show that $G_{t_0}'(b) \geq 0$ when $t_0 \leq \frac 12$. This can be written as
\[
[\sigma(b)^2]'(e^b-1)(1 - t_0\sigma(b)) - 2\sigma(b)^2 \geq 0.
\]
We see that it is enough to check this for $t_0 = \frac12$. We equivalently have
\begin{equation}\label{eq:unif_pala}
2[\sigma(b)^2]'(e^b-1) - 4\sigma(b)^2 \geq \sigma(b) [\sigma(b)^2]'(e^b-1).
\end{equation}
The left hand side is 
\[
	\frac{2 \left(e^b(b^2-6) +2b^2+6b+6\right)}{e^b-1}
\]
and since $p(b):=e^b(b^2-6) +2b^2+6b+6$ satisfies $p(0)=p'(0)=p''(0)=0$ and $p'''(b)=b(b+6)e^b \geq 0$, we have $p(b) \geq 0$ and we can square both sides of \eqref{eq:unif_pala} to get an equivalent form
\[
	(2[\sigma(b)^2]'(e^b-1) - 4\sigma(b)^2)^2 \geq \sigma(b)^2 ([\sigma(b)^2]')^2(e^b-1)^2.
\]
After multiplying by $(e^b-1)^3$ we get
\[
4 \left(e^b(b^2-6) +2b^2+6b+6\right)^2(e^b-1) \geq \left(-b^2-2 b+2 e^b-2\right) \left(e^b b^2+2 b-2 e^b+2\right)^2. 
\]
This inequality  can be written in the form
\[
P_3(b)e^{3b} + P_2(b) e^{2b} + P_1(b)e^b + P_0(b) \geq 0,
\]
this time $P_0, P_1, P_2, P_3$ are given by
\begin{align*}
    P_0(b) &= -12b^4 - 80b^3 -212b^2 - 264b - 136, \\
    P_1(b) &= 4b^5 + 12b^4 + 56b^3 + 264b^2 + 528b + 408, \\
    P_2(b) &= b^6+2 b^5+10 b^4+32 b^3-12 b^2-264 b-408, \\
    P_3(b) &= 2b^4 - 40b^2 + 136.
\end{align*}
We expand into power series $P_3(b)e^{3b} + P_2(b) e^{2b} + P_1(b)e^b + P_0(b) = \sum_{n=0}^{\infty}p_nb^n$ and once again we will see that all coefficients are non-negative. Indeed, one can manually check that for $n\leq 5$ we have $p_n = 0$. For $n \geq 6$ we have formula
\begin{align*}
n!p_n &= 4 n^5-28 n^4+124 n^3+28 n^2 +400 n+408 \\
 & \quad +3^n \left(\frac{2 n^4}{81}-\frac{4 n^3}{27}-\frac{338 n^2}{81}+\frac{116 n}{27}+136\right) \\
& \quad +2^n \left(\frac{n^6}{64}-\frac{11 n^5}{64}+\frac{85 n^4}{64}-\frac{69 n^3}{64}-\frac{223 n^2}{32}-\frac{1001 n}{8}-408\right).
\end{align*}
For $n \geq 6$ we have 
\[
4 n^5-28 n^4+124 n^3+28 n^2 +400 n+408 \geq 0.
\]
Moreover, for $n \geq 18$
\[
\frac{2 n^4}{81}-\frac{4 n^3}{27}-\frac{338 n^2}{81}+\frac{116 n}{27}+136 \geq \frac{2 n^4}{81}-\frac{4 n^3}{27}-\frac{23 n^3}{81} = \frac{n^3}{81}(2n-35)>0.
\]
Finally, for $n \geq 18$ we have $\frac{223 }{32} n^2 \leq \frac{223 }{32 \cdot 18} n^3$, $\frac{1001 }{8} n \leq \frac{1001 }{8\cdot 18^2} n^3$ and $408 \leq \frac{408}{18^3} n^3$ and thus
\begin{align*}
\frac{n^6}{64}-\frac{11 n^5}{64}+\frac{85 n^4}{64}-\frac{69 n^3}{64}-\frac{223 n^2}{32}-\frac{1001 n}{8}-408 > \frac{n^6}{64}-\frac{11 n^5}{64}+\frac{85 n^4}{64}-2n^3 > 0.
\end{align*}
We again directly verify  that $n!p_n \geq 0$ holds for $n \leq 17$.


\begin{thebibliography}{9} 

\bibitem{B74} Borell, C., Convex measures on locally convex spaces,  Arkiv för Matematik 12 (1974), 239--252.   

\bibitem{B86} Bourgain, J., On high-dimensional maximal functions associated to convex bodies, Amer. J. Math. 108 (1986), no.
6, 1467--1476.

\bibitem{BGVV14} Brazitikos, S., Giannopoulos, A., Valettas, P., Vritsiou, B-H., Geometry of isotropic convex bodies. Mathematical Surveys and Monographs, 196. American Mathematical Society, Providence, RI, 2014.

\bibitem{FG06} Fradelizi, M., and Gu\'edon, O., A generalized localization theorem and geometric inequalities for convex bodies, Adv. Math. 204 no. 2 (2006), 509--529.

\bibitem{KL25} Klartag, B., Lehec, J., Affirmative Resolution of Bourgain's Slicing Problem using Guan's Bound, Geom. Funct. Anal. (GAFA) 35, (2025), 1147--1168.

\bibitem{MNT21}
Madiman, M., Nayar, P., Tkocz, T., Sharp moment-entropy inequalities and capacity bounds for logconcave distributions, IEEE Trans. Inform. Theory 67 (2021), no. 1, 81--94.


\bibitem{MTW26} Melbourne, J., Tkocz, T., Wyczesany, K., A Rényi entropy interpretation of anti-concentration and noncentral sections of convex bodies, Commun. Contemp. Math. 28 (2026), no. 1, Paper No. 2550033, 18 pp.

\bibitem{NT23}  Nayar, P., Tkocz, T., Extremal sections and projections of certain convex bodies: a survey, Harmonic analysis and convexity, 343--390, Adv. Anal. Geom., 9, De Gruyter, Berlin, 2023.
 
\end{thebibliography}
\end{document}